\documentclass[letterpaper,11pt]{amsart}
\usepackage[margin=1.2in]{geometry}
\usepackage{amssymb}
\usepackage{graphicx}
\usepackage{amsmath}
\usepackage{amsthm}
\usepackage{xypic}
\usepackage[all]{xy}
\usepackage{flafter}
\usepackage{afterpage}
\usepackage{pdfpages}

\newcommand\p{{\mathbb{P}}}
\newcommand\sh{{\mathcal{O}}}

\newcommand\Pic[1]{\hbox{\rm Pic(}#1\hbox{\rm )}}

\newcommand\C{{\mathbb{C}}}

\newcommand\U{{\mathcal{U}}}

\newcommand\E{{\mathcal{E}}}
\newcommand\s{{\mathcal{S}}}

\newcommand\bl{{\rm Bl}}
\newcommand\Chow{{{\rm Chow}}}

\newcommand\codim{{{\rm codim}}}
\newtheorem{theorem}{Theorem}[section]

\newtheorem{corollary}[theorem]{Corollary}

\newtheorem{clam}[theorem]{Claim}
\newtheorem{prop}[theorem]{Proposition}

\theoremstyle{definition}

\newtheorem{definition}[theorem]{Definition}

\newtheorem*{ack}{Acknowledgments}
\newtheorem*{notation}{Notation}

\theoremstyle{remark}
\newtheorem{re}[theorem]{Remark}

\newtheorem{step}{Step}

    \newtheoremstyle{TheoremNum}
        {1}{1}              
        {\itshape}                      
        {}                              
        {\bfseries}                     
        {.}                             
        { }                             
        {\thmname{#1}\thmnote{ \bfseries #3}}
    \theoremstyle{TheoremNum}
    \newtheorem{thmn}{Theorem}

\begin{document}

\author    {Davide Fusi}
\address   {Department of Mathematics,
                  The Ohio State University,
                  231 W 18th Ave,
                  Columbus, OH 43210, USA}
 
\email     {fusi.1@math.osu.edu}

\title{On rational varieties of small degree}

\maketitle

\begin{abstract} We prove a stronger version of a criterion of rationality given by Ionescu and Russo. We use this stronger version to define an invariant for rational varieties (we call it rationality degree), and we classify rational varieties for small values of the invariant.
\end{abstract}

\section{Introduction}

In \cite{IR}, Ionescu and Russo gave a criterion that  describes rationality in terms of suitable covering families of rational curves passing through a fixed point $x$ contained in the smooth locus of $X$. More precisely, they proved that if there exists a family of rational curves $V$ on a projective variety $X$ such that:
\begin{itemize}
\item[$i)$] the curves parametrized by $V$ cover $X$,
\item[$ii)$] all the curves parametrized by $V$ pass through a smooth point $x$ at which are all smooth,
\item[$iii)$] the general curve parametrized by $V$ is irreducible and is uniquely determined by its tangent vector at $x$,
\end{itemize}
then $X$ is rational.

Their criterion leaves open the question whether a {\em global system of parameters} can be centered at the point $x$,  that is, if the local ring of $X$ at $x$ is $\C$-isomorphic to the local ring of a point $y \in \p^n$.
Karzhemanov in \cite{Ka} proved  that even smoothness and rationality of $X$ does not guarantee that the local ring of $X$ at any of its point is $\C$-isomorphic to the local ring of a point $y \in \p^n$. In this note, we address the question and we make their result more precise proving the following:

\begin{thmn}[\ref{cr}] Let $X$ be a projective variety of dimension $n$. Suppose that a family of rational curves $V$ satisfies the hypothesis $i)$, $ii)$ and $iii)$ above for some point $x\in X_{reg}$. Then $x$ admits a global system of parameters,  or equivalently, the local ring of $X$ at $x$ is $\C$-isomorphic to the local ring of a point $y \in \p^n$. In particular, general curves of the family  are naturally identified locally to
images of lines in $\p^n$ passing through $y$ via a birational map $X\dashrightarrow \p^n$.
\end{thmn}

We would like to point out that Theorem \ref{cr} was inspired by the criterion for rationality given in the first version of \cite{IR}.

A consequence of Theorem \ref{cr} is that any family of rational curves $V$  satisfying assumption $i)$, $ii)$ and $iii)$ above carries much information on the geometry of $X$. Namely, the normalization $U\rightarrow \mathcal{U}$ of the universal family $\mathcal{U}\rightarrow V$ is a resolution of a birational map $\phi : X \dashrightarrow \p^n$ that sends the irreducible curves parametrized by $V$ into lines of the projective space $\p^n$ (see Corollary \ref{ccr}). In particular, to any family of rational curves  $V$ as above corresponds a unique birational map $\phi^V : X \dashrightarrow \p^n$ mapping irreducible curves parametrized by $V$ into lines through a point $p$ image of $x$ (here unique is intended modulo isomorphisms of $\p^n$).

Building from this, we define an invariant that can be attached to the variety $X$ in the following way. First, for any family $V$ as above,  we define the degree of $V$ as the minimal intersection of the elements parametrized by $V$ with very ample divisors whose complete linear system contains a set of sections defining $\phi^V$ (Definition \ref{rd1}). Then, we define the {\em rationality degree} of $X$ as the minimum degree among all the families (Definition \ref{rd2}). Note that this is the natural choice when we think the variety $X\subset \p^N$ embedded. Moreover, this invariant at the same time gives us control on the number of components of the reducible curves of $V$ and implicitly carries information on the birational map $\phi^V : X \dashrightarrow \p^n$ induced by $V$. When the well understood theory of smooth rational surfaces is revisited under this prospective, it turns out that the quadric $\p^1 \times \p^1$ has rationality degree $2$ and the Hirzebruch surface $\mathbb{F}_k$ has rationality degree $k+1$. This shows how this invariant somehow measures the "birational distance" of the variety $X$ from the projective space $\p^n$.

We conclude this note by studying rational varieties for small values of the rationality degree. As one can expect, the only variety with rationality degree one is the projective space $\p^n$ (Proposition \ref{Pn}). When the rationality degree is two, we have the following. 
\begin{thmn}[\ref{thmdeg2}] Let $X$ be a smooth projective rational variety. If the rationality degree of $X$ is $2$ then $X$ is one of the following:
\begin{itemize}
\item[a)] $Q^n\subset\p^{n+1}$, the quadric hypersurface;
\item[b)] $\bl_L(\p^n)$, the blow-up of $\p^n$ along a linear subspace $L$;
\item[c)] $\p^k\times \p^h$, with $h,k\geq 1$;
\item[d)] $Z\subset\p^k\times \p^h$, where $Z$  is the general section in $|\sh(1,1)|$ and $h+k\geq 3$;
\end{itemize}
\end{thmn}
In particular, the family $V$ can be embedded in a family of conics in some projective space $\p^N$ and $X$ is a rational conic connected manifold as defined in \cite{IR}. The proof of this theorem uses the results of \cite{IR}. Cases b), c), and d) immediately follows from the classification of conic connected manifolds (Theorem 2.2, \cite{IR}). Case a) requires a deeper analysis of the base locus of the map $\phi^V$. 
Note that case a) shows that the hyperquadric $Q^n\subset\p^{n+1}$ is the only conic connected manifold with Picard number one such that the conics through its general point  $x$ are locally images of lines through a point $y\in \p^n$.

\begin{notation} Through all the paper we use standard notation as in \cite{Har77} and \cite{L}. All the varieties are defined over the field of the complex numbers $\C$. In particular, we adopt the following:\\

\noindent
 ${\rm Chow}_{rat}^1(X)$ \hspace{1cm} Chow variety parametrizing $1$-cycles with rational components.\\ 
 ${\rm Chow}_{rat}^1(X)_x$ \hspace{0.85cm} Chow variety parametrizing $1$-cycles with rational components containing\\
\hspace*{3.2cm} the point $x$.
\end{notation}

\begin{ack} I am extremely grateful to my Ph.D. advisor T. de Fernex for proposing me the problem, for his useful comments, suggestions and support.

\end{ack}

\section{Ionescu-Russo's criterion of rationality }

We start this section reproving the criterion given by Ionescu and Russo. Next we give the definition of {\em admitting a global system of parameters} for a point $x\in X$ and we prove that every point satisfying the assumptions of Theorem \ref{ir} admits a global system of parameters (Theorem \ref{cr}). We conclude the section showing how the projective space $\p^n$ can be characterized in view of Theorem \ref{ir} (Proposition \ref{Pnchar}). 

\begin{theorem}[Theorem 1.3, \cite{IR}] \label{ir} Let $X$ be a projective variety of dimension
$n$. Then $X$ is rational if and only if for some $x\in X_{reg}$ there exists an irreducible closed subscheme $V\subset{\rm Chow}_1^{rat}(X)_x$
such that:
\begin{itemize}
\item[$i)$] if $\mathcal{U}\rightarrow V$ is the universal family over $V$, then the tautological morphism $\mathcal{U}\rightarrow X$ is surjective.
\item[$ii)$] all the curves parametrized by $V$ are smooth at the point $x$.
\item[$iii)$] the general curve parametrized by $V$ is irreducible and is uniquely determined by its tangent vector at $x$.
\end{itemize}
\end{theorem}

\begin{proof}
Assume $X$ is rational. Fix a birational map $X\dashrightarrow
\p^n$. Then the family of rational curves on $X$ induced by the
lines through a general point of $\p^n$ satisfies properties $i)$, $ii)$ and $iii)$ of the statement.

Conversely, suppose that for some $x\in X_{reg}$ there exists a family of rational curves parametrized by a closed subscheme $V\subset{\rm Chow}_1^{rat}(X)_x$ as in the statement.
 Let $u:\mathcal{U}\rightarrow V$ be the universal family and $\tau
:\mathcal{U}\rightarrow X$ be the tautological morphism. By
assumptions $\tau$ is surjective and $u$ admits a section
$\mathcal{S}$, which is contracted by $\tau$ to $x$. Let ${\rm Bl}_x (X)$ be the blow-up of $X$ at the point
$x=\tau(\mathcal{S})$, with exceptional divisor $E\cong\mathbb{P}^{n-1}$. Due to the smoothness assumption, the morphism $\tau$ lifts to a morphism $\widetilde{\tau} :\mathcal{U} \longrightarrow {\rm Bl}_x(X)$. Let $\widetilde{\tau}_{\s}
:\mathcal{S}\longrightarrow E \cong \mathbb{P}^{n-1}$ be the restriction of $\widetilde{\tau}$ to $\mathcal{S}$; in particular, a morphism $\nu :V\rightarrow E$ is well defined. The following picture summarizes the situation:
\begin{equation}\label{diagram1}
\xymatrix{
\mathcal{S} \ar@/^20pt/[rrr]^{\widetilde{\tau}_{\mathcal{S}}} \ar@{^{(}->}[r] & \U \ar[d]_u \ar[r]^{\widetilde{\tau}} \ar[rd]^{\tau} & {\rm Bl}_x(X) \ar[d] & E \cong \mathbb{P}^{n-1} \ar@{_{(}->}[l] \\
&V \ar@/_25pt/[urr]_{\nu} & X
}
\end{equation}
By the surjectivity of the map $\tau$ and by condition $iii)$ of the statement, it follows that the morphism $\widetilde{\tau}_{\s}$ is birational. In particular,
$\mathcal{S}$ is rational.\\
Now consider the normalization $U$ of
$\U$. This is generically a $\mathbb{P}^{1}$-bundle over the
rational variety $V$ with a section. This shows that the variety $\U$ is rational.\\
To conclude the proof we have to show that
the morphism $\widetilde{\tau}:\U\longrightarrow {\rm Bl}_x(X)$ is
generically one to one. To this end, consider the tangent space $T_v
(\U)=T_v(\mathcal{S})\oplus (\mathcal{N}_{\mathcal{S}|
\U})_{|v}$ at a
general point $v \in \mathcal{S}$. Since the induced map ${\rm d} \widetilde{\tau}:T_v
(\U)\longrightarrow T_{\widetilde{\tau}(v)}{\rm Bl}_x(X)$
splits in the two invertible linear maps
$T_v(\s)\rightarrow T_{\widetilde{\tau}(v)}({\rm Bl}_x(X))$ and $(\mathcal{N}_{\mathcal{S}|U})_{|v}\rightarrow T_{\widetilde{\tau}(v)}({\rm Bl}_x(X))$, whose images in $T_{\widetilde{\tau}(v)}({\rm Bl}_x(X))$ intersect only at $0$, the map $\widetilde{\tau}$ does not ramify along $E$. This and the assumption that all the curves parametrized by $C$ are smooth at $x$ show that $\widetilde{\tau}$ is birational.
\end{proof}

\begin{re}  As pointed out in \cite{IR}, if one weakens the condition $ii)$ by only requiring that the general curve parametrized by $V$ is smooth at $x$, then one gets a criterion for unirationality.
\end{re}

\begin{definition} We say that a point $x$ contained in the regular locus of a rational variety $X$ \emph{admits a global system of parameters} if there exists a birational map $\phi : X\dashrightarrow \p^n$ that restricts to an isomorphism in an open neighborhood of $x$.
\end{definition}

\begin{re}  A smooth rational variety for which every point admits a global system of parameters is called {\it uniformly rational} (cfr. \cite{BB}). The question whether all smooth rational varieties are uniformly rational was first raised by Gromov in \cite{G}. Recently, Karzhemanov showed that this is not the case providing an example of a smooth rational fivefold that is not uniformly rational (cfr. \cite{Ka}).

\end{re}

\begin{theorem}\label{cr} Let $X$ be a projective variety of dimension $n$. Suppose that a family of rational curves $V$ satisfies the hypotheses of Theorem \ref{ir} at some point $x\in X_{reg}$. Then $x$ admits a global system of parameters,  or equivalently, the local ring of $X$ at $x$ is $\C$-isomorphic to the local ring of a point $y \in \p^n$. In particular, general curves of the family naturally identify locally to
images of lines in $\p^n$ passing through $y$.
\end{theorem}

\begin{proof} We keep the same notation as in the proof of Theorem~\ref{ir}. Let $U\rightarrow \U$ be the normalization of $\U$. Let  $g : V'\rightarrow  V$ be a resolution of singularities and let $Z:= U\times_V V' $ be the variety obtained by base change.

\begin{equation}
\xymatrix{%
 Z:= U\times_V V'  \ar[d]^{p}  \ar[r] &        U \ar[dr]  \ar[d]  \ar[r]& \bl_x(X)\ar[d]\\
                 V'         \ar[r]^g          &      V &  X
}
\end{equation}
 
Note that $Z$ is irreducible. If $V^0\subset V'$ is the open subset where $g$ restricts to an isomorphism and $U^0$ is the universal family over $g(V^0)$, then $U_0$ embeds in an irreducible component $Z_1$ of $Z$. By universal property of the fibred product, the projection $p : Z\rightarrow V'$ admits a section $\E'\cong V'$, hence irreducible. Note that $\E'$ is contained in $Z_1$. For any point $v\in V'$ the fiber $p^{-1}(v)$ is isomorphic to the rational 1-cycle parametrized by the point $g(v)\in V$. We need to show that the fiber at any point $v\in V'$ belongs to $Z_1$.  Fix a point $v\in V'$. Let $C$ be a smooth curve through $v$ such that $C$ is not contracted by $g$ and such that $g(u)$ parametrizes an irreducible cycle if $u$ is general in $C$. Let $S$ be the surface over $C$ obtained by base change from the universal family over $g(C)$. Since the surface $S$ is irreducible and $S$ embeds in $Z_1$, it follows that $Z=Z_1$.
 
Note also that $\E'$ is the pull-back on $Z$ of the exceptional divisor $E\subset \bl_x(X)$. In particular, $Z$ is smooth along $\E'$ (being $\E'$ a smooth Cartier divisor). Let $Y\rightarrow Z$ be a resolution of singularities that is an isomorphism in a neighborhood of $\E'$ and let $\E$ be the pull-back on $Y$ of $\E'$.

Let us fix some notation. Let $\psi: Y \rightarrow X$ (rest. $\widetilde{\psi}: Y \rightarrow \bl_x(X)$) be the morphism induced by the tautological morphism $\tau: \U \rightarrow X$ (resp. by the lift of the tautological morphism $\widetilde{\tau}: \U \rightarrow \bl_x(X)$). Let $p: Y \rightarrow V'$ be the natural projection and let $v: V'\rightarrow \p^{n-1}$ be the morphism induced by $\nu : V\rightarrow \p^{n-1}$.

We want to construct a birational morphism $\varphi:Y\rightarrow\p^n$ such that: 
\begin{itemize}
\item[i)] the section $\mathcal{E}$ is contracted to a point $p\in\p^n$; 
\item[ii)] if we look at the morphism $\widetilde{\varphi}_{\mathcal{E}}:\mathcal{E}\rightarrow F\cong\mathbb{P}^{n-1}\subset{\rm Bl}_p(\p^n)$ (from $\mathcal{E}$ to the exceptional divisor $F$ of the blow-up of $\p^n$ at the point $p$, induced by the lifting $\widetilde{\varphi}:Y\rightarrow {\rm Bl}_p(\p^n)$), then $\widetilde{\varphi}_{\E}=\widetilde{\psi}_{\mathcal{E}}$ modulo isomorphisms of $E$ and $F$; 
\item[iii)] the morphism $\varphi$ maps every irreducible component of the fibers of $Y\rightarrow V'$, that intersects $\E$, isomorphically to a line.
\end{itemize}
 
\begin{clam} A morphism $\varphi$ satisfying properties $i)$, $ii)$ and $iii)$ induces a birational map  $\phi : X\dashrightarrow\p^n$ 
\begin{equation}\label{diagram2}
\xymatrix{%
 {\rm Bl}_p(\p^n) \ar[d]& \ar[l]^{\widetilde{\varphi}}Y\ar[rd]^{\psi}\ar[ld]^{\varphi}  \ar[r]^{\widetilde{\psi}}& {\rm Bl}_x(X)\ar[d]  \\
\p^n &  & \ar@{-->}[ll]^{\phi}  X
}
\end{equation}
that is an isomorphism in an open neighborhood of $x$.
\end{clam}
\begin{proof} Let $\mathcal{H}$ be the complete linear system $H^0(\p^n,\sh (1))$. Note that the map $\phi :X \dashrightarrow \p^n$ is defined by the birational transform $\phi_*^{-1}(\mathcal{H})$ of $\mathcal{H}$. Note that $\phi_*^{-1}(\mathcal{H})=\psi_*(\varphi^* (\mathcal{H}))$, here we use that $\mathcal{H}$ is globally generated. Since $\varphi^*(\mathcal{H})$ contains a divisor disjoint from $\E$ and the inverse image $\psi^{-1}(x)$ of $x$ is $\E$, the map $\phi$ is defined at $x$ and $\phi (x)=p$. Looking at the birational transform on $\p^n$ of the complete linear system $H^0(X,\sh_X (1))$, we get that $\phi^{-1}$ is defined at $p$ and $\phi^{-1}(p)=x$. In particular, $\phi$ is an isomorphism in a neighborhood of $x$.
\end{proof}

The rest of the proof is dedicated to defining the morphism $\varphi$. Let $H$ be a general hyperplane in $E$ and $D=p^*(v^*(H))$ be the pull back on $Y$ of $H$. We want to show that the morphism $\varphi$ exists and it is given by the complete linear system $|D + \mathcal{E}|$. We proceed by steps.
\begin{step}  We want to show that $\sh_{\E} (D+\E)\cong\sh_{\E}$.

To this end consider the following commutative diagram:
$$
\xymatrix{%
V'\cong \E \ar[r]^{\widetilde{\psi}_{\mathcal{E}}} \ar[d]^{j} & E\cong \p^{n-1}\ar[d]^{i}\\
Y \ar[r]^{\widetilde{\psi}} & {\rm Bl}_x(X)
}
$$
where $i$ and $j$ are the natural inclusions. Observe that $\sh_{\E}(\E)\cong j^*\sh_{Y}(\E)\cong j^*(\widetilde{\psi}^*\sh_{{\rm Bl}_x(X)}(E))$. By the commutativity of the diagram and the functoriality of the pull-back in the morphisms, we get that $j^*(\widetilde{\psi}^*\sh_{{\rm Bl}_x(X)}(E))\cong \widetilde{\psi}_{\E}^*(i^*\sh_{{\rm Bl}_x(X)}(E))$. Recalling that $\sh_{\E}(D)\cong \widetilde{\psi}_{\E}^*\sh_E(H)$, we get that
$$\sh_{\E} (D+\E)\cong \widetilde{\psi}_{\E}^*\sh_E(H+E)\cong \sh_{\E}.$$
\end{step}

\begin{step} We want to show that $H^1(Y,\sh_{Y}(\mathcal{E}))=0$.

Using the projection formula, we have:
$$\widetilde{\psi}_{\mathcal{E}*}(\sh_{\E}\otimes\widetilde{\psi}_{\mathcal{E}}^*\sh_{\mathbb{P}^{n-1}}(-1))=\widetilde{\psi}_{\mathcal{E}*}\sh_{\E}\otimes \sh_{\mathbb{P}^{n-1}}(-1).
$$
Using projection formula for higher image sheaves, we get:
$$R^i(\widetilde{\psi}_{\mathcal{E}*}(\sh_{\E}\otimes\widetilde{\psi}_{\mathcal{E}}^*\sh_{\mathbb{P}^{n-1}}(-1)))=R^i\widetilde{\psi}_{\mathcal{E}*}\sh_{\E}\otimes \sh_{\mathbb{P}^{n-1}}(-1)=0
$$
for $i>0$, because $R^i\widetilde{\psi}_{\mathcal{E}*}\sh_{\E}=0$ ($E$ being smooth, and hence only with rational singularities). Recalling that $\sh_{\mathcal{E}}(\sh_{\E}\otimes\widetilde{\psi}_{\mathcal{E}}^*\sh_{\mathbb{P}^{n-1}}(-1))=\sh_{\mathcal{E}}(-D)=\sh_{\mathcal{E}}(\mathcal{E})$,
we get that $H^1(\mathcal{E},\sh_{\mathcal{E}}(\mathcal{E}))=H^1(\mathbb{P}^{n-1},\sh_{\mathbb{P}^{n-1}}(-1))=0$ (cfr. Exercise $8.1$, Chapter III of \cite{Har77}).

If we look at the long exact sequence in cohomology induced by the following short exact sequence:
$$0\rightarrow \sh_{Y}\rightarrow \sh_{Y}(\mathcal{E})\rightarrow \sh_{\mathcal{E}}(\mathcal{E})\rightarrow 0,
$$
we get that $H^1(Y,\sh_{Y}(\mathcal{E}))=0$ (since $H^1(Y,\sh_{Y})=0$, $Y$ being smooth and rational, and since $H^1(\mathcal{E},\sh_{\mathcal{E}}(\mathcal{E}))=0$, from previous discussion).\\
\end{step}
\begin{step} We want to show that $H^0(Y,\sh_{Y}(D+\mathcal{E}))=n+1$. In particular, the map $H^0(Y,\sh_{Y}(D+\mathcal{E}))\rightarrow H^0(C,\sh_{C}(D+\mathcal{E}))$ is surjective, where $C\cong \p^1$ is the general fiber of $p : Y\rightarrow V'$.

Since $\sh_{Y}(D)$ is globally generated, by Bertini's Theorem we can choose a general section $D^1\in |D|$ that is smooth. By construction a smooth section of $|D|$ has to be irreducible. From Step 2 and from the piece of long exact sequence induced in cohomology by the following short exact sequence:
$$
\xymatrix{%
0\ar[r] & \sh_{Y}(\mathcal{E}) \ar[r]^{\cdot D^1} & \sh_{Y}(D+\mathcal{E})\ar[r] & \sh_{D^1}(D+\mathcal{E})\ar[r] & 0, \\
}
$$
we get that
$$H^0(Y,\sh_{Y}(D+\mathcal{E}))=H^0(D^1,\sh_{D^1}(D+\mathcal{E}))+1.$$
Replacing $Y$ with $D^1$, $\mathcal{E}$ with $\mathcal{E}_{|D^1}$, ${\rm Bl}_x(X)$ with $\widetilde{\psi}(D^1)$ and $E$ with $E_{\widetilde{\psi}(D^1)}$, and using the same argument as above we get that
$$H^0(D^{1},\sh_{D^{1}}(D+\mathcal{E}))=H^0(D^{2},\sh_{D^{2}}(D+\mathcal{E}))+1,$$
where $D^2$ is the general (hence smooth and irreducible) section of $H^0(D^{1},\sh_{D^{1}}(D+\mathcal{E}))$.\\
Proceeding by induction on the dimension and observing that $n-1$ general sections of $H^0(Y,\sh_{Y}(D))$ intersect in a general fiber $C$ of $p : Y\rightarrow V'$, we get that:
$$H^0(Y,\sh_{Y}(D+\mathcal{E}))=H^0(C,\sh_{C}(D+\mathcal{E}))+n-1=n+1,$$
being $C\cong \p^1$ and $\sh_{C}(D+\mathcal{E})\cong \sh_{\p^1}(1)$.\\
\end{step}
\begin{step} We want to show that the linear system $|D+\mathcal{E}|$ defines the morphism we are looking for.

Without loss of generality we can choose $<ed_1,\ldots{},ed_n,s>$ as set of generators for $H^0(Y,\sh_{Y}(D+\mathcal{E}))$; where $<d_1,\ldots{},d_n>$ is a generating set for $H^0(Y,\sh_{Y}(D))$, $e \in H^0(Y,\sh_{Y}(\mathcal{E}))$ and $s\in H^0(Y,\sh_{Y}(D+\mathcal{E}))$ is a section not vanishing on $\mathcal{E}$.\\
It is clear that the set of sections $<ed_1,\ldots{},ed_n,s>$ defines a morphism ${\varphi}_{|D+\mathcal{E}|}: Y\rightarrow \p^n$ that contracts the divisor $\mathcal{E}$ to a point. Moreover there exists a lift $\widetilde{\varphi}_{|D+\mathcal{E}|}: Y\rightarrow {\rm Bl}_p(\p^n)$, where $p=\varphi_{|D+\mathcal{E}|}(\mathcal{E})$. To conclude the proof we need to show that $\varphi_{|D+\mathcal{E}|}$ is generically one to one.\\
By Step 3, the morphism $\varphi_{|D+\mathcal{E}|}$ restricted to the general fiber $C$ of $p : Y\rightarrow V'$ is an isomorphism (more precisely it sends $C$ in a line of $\p^n$). Moreover, if we call $\varphi_{|D|}: Y\rightarrow \p^{n-1}$ the morphism associated to the set of sections $<d_1,\ldots{},d_n> \in H^0(Y,\sh_{Y}(D))$, we have the following commutative diagram:
$$
\xymatrix{%
 Y \ar[rr]^{\widetilde{\varphi}_{|D+\mathcal{E}|}} \ar[dr]^{\varphi_{|D|}}&      & {\rm Bl}_p(\p^n) \ar[dl]  \\
       & \p^{n-1} &
}
$$
i.e., the morphism $\widetilde{\varphi}_{|D+\mathcal{E}|}$ is defined over $\p^{n-1}$. This shows that $\varphi_{|D+\mathcal{E}|}$ is birational and it concludes the proof.\end{step}\end{proof}

\begin{corollary}\label{ccr} If $V$ is a family of rational curves as in Theorem \ref{cr}, then $V$ determines a unique (modulo isomorphisms of $\p^n$) birational map $\phi^V:X\dashrightarrow \p^n$ that maps the general curve parametrized  by $V$ into a line through a point $p=\phi^V(x)$. Moreover, the normalization $U\rightarrow \U$ of the universal family $\U\rightarrow V$ is a resolution of the birational map $\phi^V$.
\end{corollary}

\begin{proof} Let $\phi^V:X\dashrightarrow \p^n$ be the birational map defined in the proof of Theorem \ref{cr}. To see that $U\rightarrow V$ is a resolution of $\phi^V$, it is enough to observe that the linear system $|D+\E |$, defining the morphism $\varphi : Y\rightarrow \p^n$, is the pull-back of a linear system on $U$. 

Assume that there exists a birational map $\phi' :X\dashrightarrow\p^n$ that maps the general curve parametrized by $V$ into a line through a point $p'=\phi'(x)$.
$$
\xymatrix{%
      \p^n \ar@{-->}@/^20pt/[rr]_{f} & X \ar@{-->}[l]^{\phi^V} \ar@{-->}[r]_{\phi'} & \p^n & \\
}
$$
The two birational morphisms $\phi^V$ and $\phi'$ induce a birational map $f : \p^n\dashrightarrow \p^n$. Since the general line through the point $p$ is sent to a line through the point $p'$, the map $f$ is an isomorphism. This concludes the proof of the Corollary.
\end{proof}

\begin{prop}\label{Pnchar} Let $X$ be a normal projective variety of dimension $n$. Then $X\cong \p^n$ if and only if there exists a family of rational curves
$V\subset {\rm Chow}_1^{rat}(X)_x$ satisfying the hypothesis of Theorem \ref{ir} at some point $x\in X_{reg}$ such that all the members of $V$ are irreducible.
\end{prop}

\begin{proof} The "if part" is trivial. If $X\cong \p^n$, consider the family of lines through a point $x\in \p^n$.

The following claim is the main ingredient to prove the other implication.

\begin{clam} If all the curves $C$ of the family $V$ are irreducible then $V\cong\mathbb{P}^{n-1}$.
\end{clam}

\begin{proof} The proof is based on a useful remark due to Kebekus, (\cite{kebekus}, proof of Theorem $3.4$), pointed out by the two authors in \cite{IR}.
Using the same notation as in the proof of Theorem~\ref{ir}, suppose that the morphism $\nu :V\rightarrow E$ contracts a curve $B\subset V$ to a point $e$ of $E\cong\p^{n-1}$. Take $\overline{B}$, the normalization of $B$, and consider the surface $S$ over $\overline{B}$ obtained by base change from the universal family over $B$. Due to the assumption, the projection $p : S \rightarrow \overline{B}$ admits a section $D$; the surface $S$ is irreducible and smooth along $D$. It is well defined a morphism $f :S \rightarrow X$, via composition with the tautological morphism from $\U$ to $X$. The tangent morphism of $f$ induces a morphism $f' : \mathcal{N}_{D|S}\rightarrow l_e$, where $\mathcal{N}_{D|S}$ is the geometric normal bundle of $D$ in $S$ and $l_e$ is the line in the tangent space of $X$ at $x$ corresponding to the point $e \in E$. This tells us that the normal bundle $\mathcal{N}_{D|S}$ has to be trivial. Indeed, if $\mathcal{N}_{D|S}$ is not trivial, the map $f'$ has a zero and the corresponding curve in the family is singular at the point $x$, against the assumptions.
Since $\mathcal{N}_{D|S}$ is trivial the map $f: S\rightarrow X$ is a morphism over a curve $B'$. Since all the curves parametrized by $V$ are irreducible, then all the points of $B$ parametrize the same $1$-cycle, a contradiction.
\end{proof}

Without loss of generality we can replace the universal family $\U$ with its normalization $U\rightarrow \U$. Since all the curves parametrized by $V$ are irreducible, the universal family $u :U \rightarrow V$ is a $\p^1$-bundle over $V$. Indeed, $U$ is a resolution of the birational map $\phi^V : X\dashrightarrow \p^n$ defined in Corollary \ref{ccr}: 
$$
\xymatrix{%
    & U \ar[dr]^{\pi} \ar[dl]  &             \\
    X \ar@{-->}[rr]^{\phi^V} &  & \p^n\\
}
$$
Since $V\cong \p^{n-1}$, the curves parametrized by $V$ are in a one to one correspondence with the lines in $\p^n$. Any curve $C$ parametrized by $V$ is irreducible and the morphism $\pi : U \rightarrow \p^n$  when restricted to $C$ defines an isomorphism with the corresponding line in the projective space $\p^n$.

The section $\mathcal{S}$ is contracted to a smooth point by the tautological morphism $\tau:U\rightarrow X$, hence $\sh_\mathcal{S}(\mathcal{S})\cong \sh_{\mathcal{S}}(-1)$, and $U\cong \mathbb{P}_{\mathbb{P}^{n-1}}(\sh \oplus\sh (-1))$. This means that $U$ is isomorphic to the blown up of $\p^n$ at one point. Now consider the map $\tau : U \rightarrow X\subset\p^N$, where we think the variety $X$ embedded in some projective space $\p^N$. The map $\tau$ is given by some linear system $\Lambda\subset H^0(U, L)$, where $L$ is a divisor on $U$. Hence we can write $L\cong \widetilde{rH}+s\mathcal{S}$ for some  $r, s\in \mathbb{Z}$, where $\widetilde{H}$ is the pull back of the hyperplane section of $\p^n$ via the contraction morphism $U\rightarrow \p^n$. Since $\mathcal{S}$ is contracted by $\tau$, it follows that $L_{\mathcal{S}}\cong \sh_{\mathcal{S}}$. In particular, we have that $s=0$ and $r\geq 1$. The morphism $\tau$ factors through a morphism $t:\p^n\rightarrow X$:
$$
\xymatrix{%
U\ar[dr] \ar[rr]^{\tau} &  & X\subset\p^n\\
    & \p^n\ar[ur]^{t}.  &             \\
}
$$
Since ${\rm Pic} (\p^n) \cong \mathbb{Z}$ then $t$ is a finite birational morphism. It follows by the normality of $X$ that $t$ is actually an isomorphism. This concludes the proof.
\end{proof}

\begin{re} Theorem \ref{cr} and Proposition \ref{Pnchar} suggest that the closed subschemes $V\subset {\rm Chow}_1^{rat}(X)_x$ as in Theorem \ref{ir} carry more information about the geometry of $X$ than we can actually understand at this time. Indeed, to such a $V$ corresponds in a natural way a birational map $\phi^V:X\dashrightarrow\mathbb{P}^n$ as in Corollary \ref{ccr}. If the curve $C_v$ parametrized by a point $v\in V$ is reducible, then we can write $C_v=C_{v,x}\cup Z_v$, where $C_{v,x}$ is irreducible and $x\in C_{v,x}$. Theorem \ref{cr} shows that $\phi^V(C_{v,x})$ is a line in $\mathbb{P}^n$ and that $\phi(Z_v)$ is contained in the indeterminacy locus of $(\phi^V)^{-1}$. In the next sections we give a notion of minimality for the family of rational curves $V$ and some accessible examples of minimal families are studied.  It turns out that in all these first examples the component $C_{v,x}$ of $C_v$ (where $v\in V$ parametrizes a reducible curve) is a minimal rational curve and $\nu(v)$ is a point of the variety of minimal rational tangents at $x$ (cfr. \cite{H1}, \cite{H2}). It would be interesting to understand if this is always the case, that is, if the family $V$ is in some sense minimal (non necessarily minimal under our definition) this implies that there exists a family of minimal rational tangents  $H$ such that $H_x$ is contained in $\bigcup_{C_v red.} C_{v,x}$.
\end{re}

\endproof

\section{Rationality Degree}

In this section, as an application of Theorem \ref{cr}, we give a possible definition of rationality degree for a projective rational variety $X$ and we classify rational varieties for small values of the rationality degree. 

We saw in the previous section: if a variety $X$ is rational,
then for the general point $x\in X$ we can find a family of rational
curves parametrized by a closed subscheme $V\subset
{\Chow}_{rat}^1(X)_x$ satisfying the assumption of
Theorem~\ref{ir}. The family $V$
induces a birational map $\phi^V :X\dashrightarrow
\p^n$. In particular, the normalization $U$ of the universal family $\U\rightarrow X$ is a resolution of $\phi^V$. Diagram~\ref{diagram1} and Diagram~\ref{diagram2} summarize the setting we will work and the notation we will use through this section.

Diagram~\ref{diagram1}:
$$
\xymatrix{%
\mathcal{E} \ar@/^20pt/[rrr]^{\widetilde{\tau}_{\mathcal{E}}} \ar@{^{(}->}[r] & U \ar[d]_u \ar[r]^{\widetilde{\tau}} \ar[rd]^{\tau} & {\rm Bl}_x(X) \ar[d] & E \cong \mathbb{P}^{n-1} \ar@{_{(}->}[l] \\
&V \ar@/_25pt/[urr]_{\nu} & X
}
$$
Diagram~\ref{diagram2}:
$$
\xymatrix{%
 {\rm Bl}_p(\p^n) \ar[d]& \ar[l]^{\widetilde{\varphi}}U\ar[rd]^{\tau}\ar[ld]^{\varphi}  \ar[r]^{\widetilde{\tau}}& {\rm Bl}_x(X)\ar[d]  \\
\p^n &  &  \ar@{-->}[ll]^{\phi} X
}
$$
In the above diagrams we kept the notation as in Theorem~\ref{ir} and Theorem~\ref{cr}.

\begin{re}Looking at the Cremona transformations of the projective space $\p^n$ to itself, one can easily see that, given a rational variety $X$, for the general point $x \in X_{reg}$ there exist infinitely many families of rational curves of arbitrarily large degree satisfying the assumption of Theorem~\ref{ir}.
\end{re}

\begin{re}\label{famcha} Let $V\subset {\Chow}_{rat}^1(X)_x$ be a family of rational curves satisfying the assumption of Theorem~\ref{ir} at a point $x\in X_{reg}$. It is easy to see that for the general $x'\in X_{reg}$ there exists a family of rational curves $V'\subset {\Chow}_{rat}^1(X)_{x'}$, satisfying the assumption of Theorem~\ref{ir} at the point $x'\in X$, such that:\\
i) $V$ and $V'$ are in the same connected component of ${\Chow}_{rat}^1(X)$;\\
ii)  $V$ and $V'$ induce the same birational map $\phi^V=\phi^{V'} : X\dashrightarrow \p^n$. 
\end{re}

Now we are ready to define the \emph{rationality degree} of $X$.

\begin{definition}\label{rd1} Let $X$ be a projective rational variety and $V\subset {\Chow}_{rat}^1(X)_x$ be a family of rational curves satisfying the assumption of Theorem~\ref{ir} at some point $x\in X_{reg}$. Let $\phi^V: X \dashrightarrow \p^n$ be the birational map induced by $V$. Define the following set:
$$A_{V}:=\{ D \ {\rm very} \ {\rm ample} | \exists W \subset H^0(X,D){\rm \ such\ that} \  \phi_{ | W | }=\phi^V   \}.$$ 
where we allow $ | W |$ to have a fixed component. Then we call degree of $V$ the minimum of the following set: 
$$\{ D \cdot C |\ D \in A_V\ {\rm and} \ [C]\in V   \}.$$
\end{definition}

\begin{re} The set $A_V$ is never empty. Let $\phi: X \dashrightarrow \p^n$ be defined by the subspace $ W \subset H^0(X,D)$, where $D$ is a big divisor. Then $A=D+B$ for some very ample divisors $A$ and $B$. For any section $s\in H^0(X,B)$, we have the injetive homomorphism:
$$
\xymatrix{%
 \mu: H^0(X,D) \ar[r]^{\cdot s}  & H^0(X,A)\\
}
$$
given by multiplication. It is easy to see that $\phi =\phi_{ | W | }=\phi_{ | \mu ( W ) | }$.

Clearly, if $A\in A_V$ and $W \subset H^0(X,A)$ such that $\phi_{ | W | }=\phi^V$, then $\phi_{ | W | }$ maps the general curve of $V$ into a line.
\end{re}

\begin{definition}\label{rd2} Let $X$ be a rational variety.
\begin{itemize}
\item[1)] If $x\in X_{reg}$ admits a global system of parameters, then we call rationality degree at $x$ the minimum degree among all closed subschemes  $V\in {\rm Chow}_1^{rat}(X)_x$ satisfying the assumption of Theorem~\ref{ir}.
\item[2)] We call rationality degree of $X$ the minimum rationality degree among all the points $x\in X_{reg}$ that admit a global system of parameters.
\end{itemize}
\end{definition}

The next Proposition is straightforward, but it shows how the notion of rationality degree well adapts to the well known classification of smooth algebraic surfaces.

\begin{prop}\label{clsu} Let $X$ be a minimal rational surface. Then the following hold:
\begin{itemize}
\item the rationality degree of $X$ is $1$ if and only if $X\cong\p^2$;
\item the rationality degree of $X$ is $2$ if and only if $X\cong\p^1\times\p^1$;
\item the rationality degree of $X$ is $k+1$ ($k>1$) if and only if $X\cong\mathbb{F}_k$.
\end{itemize}
\end{prop}

\begin{proof} The statement is clear when $X\cong \p^2$ and when $X\cong \p^1\times \p^1$.

Let $X$ be a $\mathbb{F}_k$ surface with $k>1$. Let $f$ be the general fiber of the projection $\mathbb{F}_k\rightarrow \p^1$ and let $C_0$ be the section with negative self intersection, say $C_0^2=-k$. It is well known that $\overline{NE}(X)$ is generated by the classes $[f]$ and $[C_0]$.

The surface $\mathbb{F}_k$ can be connected to $\p^2$ via one link of type $I$ and $(k-1)$ links of type $II$ of the Sarkisov program (cfr. \cite{C}). Looking at the family $V$ of rational curves on $X$ induced by the lines through the general point of $\p^2$, we get that the rationality degree of $X$ is at most $k+1$ (because the general curve $C$ parametrized by $V$ has class $C_0+kf$ and the very ample divisor with class $C_0+(k+1)f$ is in $A_V$) .

Every effective curve on $X$ is numerically equivalent to $C_0$ or $aC_0+bf$, with $a\geq0$ and $b\geq ak$. Let $V$ be a family of rational curves satisfying the assumption of Theorem \ref{cr} at some point $x\in X$, $C=aC_0+bf$ be the general curve parametrized by $V$ and $D=\alpha C_0+\beta f$ be an ample divisor. Note  that $a\geq 1$ and $b\geq k$, since $V$ is a covering family and $x\in C$ for any $[C]\in V$. Similarly, $\alpha \geq 1$ and $\beta > k$, being $D$ ample. Rewriting $C$ as 
$$C=(C_0+kf)+(a-1)C_0+(b-k)f$$ 
and computing the intersection number, we have:
$$D\cdot C=\beta +\alpha (b-ak)+\beta (a-1) > k.$$
This concludes the proof.
\end{proof}

\begin{re}\label{subl} It is not hard to show that the only smooth non minimal rational surface of degree 2 is the plane $\p^2$ blown-up at one point.
\end{re}

The following obvious proposition gives a description of the rational varieties with rationality degree one. 
\begin{prop}\label{Pn} Let $X$ be a normal projective variety of dimension $n$. If the rationality degree of $X$ is $1$, then $X$ is isomorphic to the projective space $\p^n$.
\end{prop}

\begin{proof} There exists a family of rational curves $V\subset {\Chow}_{rat}^1(X)_x$ satisfying the assumption of Theorem \ref{ir} at some  point $x\in X_{reg}$. The induced birational map $\phi^V : X\dashrightarrow \p^n$ is given by a linear system $\Gamma \subset | A |$, where $A$ is very ample and $A\cdot C=1$ for $[C]\in V$. Proposition \ref{Pnchar} applies, since all the curves parametrized by $V$ are irreducible. This concludes the proof. 
\end{proof}

 The study of rational varieties with fixed rationality degree $k$ is related to the understanding of the base locus of the birational map $\phi^V:X\dashrightarrow \p^n$ induced by the family of rational curves $V$ realizing the rationality degree. When the rationality degree has small value, it imposes restrictions to the base locus of the birational map. 
In Theorem \ref{thmdeg2} we classify smooth rational varieties of rationality degree $2$.

\begin{theorem}\label{thmdeg2} Let $X$ be a smooth projective rational variety. If the rationality degree of $X$ is $2$ then $X$ is one of the following:
\begin{itemize}
\item $Q^n\subset\p^{n+1}$, the quadric hypersurface;
\item $\bl_L(\p^n)$, the blow-up of $\p^n$ along a linear subspace $L$;
\item $\p^k\times \p^h$, with $h,k\geq 1$;
\item $Z\subset\p^k\times \p^h$, where $Z$  is the general section in $|\sh(1,1)|$ and $h+k\geq 3$;
\end{itemize}
\end{theorem}

\begin{proof} There exists a family of rational curves $V\subset {\Chow}_{rat}^1(X)_x$ satisfying the assumption of Theorem \ref{ir} at some  point $x\in X$. The induced birational map $\phi^V : X\dashrightarrow \p^n$ is given by a linear system $\Gamma \subset | A |$, where $A$ is very ample and $A\cdot C=2$ for $[C]\in V$. In particular, the complete linear system $| A |$ exhibits $X$ as a conic connected manifold in $\p^{h^0(X,A)-1}$ as defined in \cite{IR}.

By Theorem 2.2 in \cite{IR}, we have that $X$ must be one of the following:
\begin{itemize}
\item[i)] a Fano manifold of Picard number one and index $i(X)\geq \frac{n+1}{2}$;
\item[ii)] $\bl_L(\p^n)$, the blow-up of $\p^n$ along a linear subspace $L$;
\item[iii)] $\p^k\times \p^h$, with $h,k\geq 1$;
\item[iv)] $Z\subset\p^k\times \p^h$, where $Z$  is the general section in $|\sh(1,1)|$ and $h+k\geq 3$;
\end{itemize}
It is easy to see that ii), iii) and iv) are rational varieties of rationality degree 2. By Proposition \ref{Pn}, it is enough to produce a family of rational curves $V$ as in Theorem \ref{ir} and an ample divisor $A$ belonging to the set $A_V$ (defined in Definition \ref{rd1}) with $A\cdot C=2$ (here $C$ is the general curve of $V$).

Let $X$ be as in ii), that is, $X$ is the blow-up $\pi : \bl_L(\p^n)\rightarrow \p^n$ of $\p^n$ along a linear subspace $L$. Call $E$ the exceptional divisor and $\mathcal{H}$ the pull-back of the hyperplane section. Note that  the contraction morphism $\pi : \bl_L(\p^n)\rightarrow \p^n$ is naturally defined by the vector space $H^0(X,\mathcal{H})$.  Let $V$ be the family of rational curves induced on $X$ by a family of lines in $\p^n$ through a point $x\not\in L$. The divisor $A=2 \mathcal{H}-E$ is very ample and $A \cdot C=2$ (note that the general curve $C$ parametrized by $V$ is the proper transform of a line). If $s\in H^0(X, \mathcal{H}-E)$ is a section whose zero locus defines the proper transform of an hyperplane containing $L$, then the homomorphism given by multiplication by $s$
$$
\xymatrix{%
 \cdot s : H^0(X,\mathcal{H}) \ar[r]  & H^0(X,A)\\
}
$$
shows that $A\in A_V$.

Assume that $X\cong \p^k\times \p^h$, i.e. $X$ is as in iii). In this case $X$ can be obtained by blowing up $\p^{h+k}$ along two disjoint linear subspaces, say $L_1\cong \p^{h-1}$ and $L_2\cong \p^{k-1}$, and blowing down the proper transform of the hyperplane $H_{L_1,L_2}$ containing  $L_1$ and $L_2$. Call $\phi : \p^{h+k}\dashrightarrow X$ the birational map described in the previous period. If we think $X$ embedded in $\p^{(h+1)(k+1)-1}$ using the Segre embedding, then $\phi$ is defined by the space of sections $H^0(\p^{h+k}, \sh (2)\otimes \mathcal{I}_{L_1}\otimes \mathcal{I}_{L_2})$. Let $V$ be the family of rational curves induced on $X\subset \p^{(h+1)(k+1)-1}$ by a family of lines in $\p^{h+k}$ through a point $x\not\in H_{L_1,L_2}$. Note that $V$ is family of conics. To conclude that the rationality degree of $X$ is 2 we need to show that the inverse map $\phi^{-1}: X \dashrightarrow \p^{h+k}$ is defined by a subspace of $H^0(X, \sh (1))$. Note that we can choose a basis of $H^0(\p^{h+k}, \sh (2)\otimes \mathcal{I}_{L_1}\otimes \mathcal{I}_{L_2})$ such that the first $h+k+1$ sections are the images of the sections of $H^0(\p^{h+k}, \sh (1))$ via the multiplication map
$$
\xymatrix{%
 \cdot s : H^0(\p^{h+k}, \sh (1)) \ar[r]  & H^0(\p^{h+k}, \sh (2)\otimes \mathcal{I}_{L_1}\otimes \mathcal{I}_{L_2})  \\
}
$$
where $s\in H^0(\p^{h+k}, \sh (1))$ is the section whose zero locus defines $H_{L_1,L_2}$. Call $\pi$ the projection from the subspace of $\p^{(h+1)(k+1)-1}$ defined by the equations $\{x_{0}=...=x_{h+k}=0\}$ to the subspace defined by the equations $\{x_{h+k+1}=...=x_{(h+1)(k+1)-1}=0\}$. It is not hard to see the $\pi$ restricted to $X$ is equal to $\phi^{-1}$.

Let $X$ be as in iv). We can think $X$ as the general hyperplane section of $\p^k\times \p^h\subset \p^{(h+1)(k+1)-1}$. In particular, if $\phi : \p^{h+k}\dashrightarrow  \p^k\times \p^h$ is the birational map described above, then $X$ is the proper transform of a hypersurface $Y$ of degree 1 or 2 in $\p^{h+k}$. It is easy to see that, if $Y$ is an hyperplane, then $X$ has rationality degree 2. Assume that $Y$ has degree 2. If $y\in Y$ is a smooth point of $Y$, then projecting $Y$ from $y$ to a general hyperplane gives a birational map $\pi_1 : Y\dashrightarrow \p^{h+k-1}$. Then the map $f :=\phi^{-1}\circ \pi_1^{-1}: \p^{h+k-1} \dashrightarrow X$ is birational. It is not hard to see that if $p$ is a point in the open subset of $\p^{h+k-1}$ where $f$ restricts to an isomorphism, then the lines through $p$ induce a family $V$ of conics on $X$. Since the birational map $f^{-1} : X\dashrightarrow \p^{h+k-1}$ is defined by a subspace of $H^0(X, \sh (1))$, it follows that $X$ has rationality degree 2.

We want to show that if $X$ is as in i), then $X\cong Q^n\subset \p^{n+1}$. Let us assume that $\Pic X\cong \mathbb{Z}$. Let $Z$ be the base locus of $\phi^V: X\dashrightarrow \p^n$. Note that if a curve $C$ in the family $V$ is reducible, then $C$ has exactly two components.
\begin{clam}\label{c1} The general divisor of the linear system $\Gamma$ is smooth. In particular, all the curves parametrized by $V$ are smooth along $Z$. 
\end{clam}

\begin{proof}  Let $\overline{D}\in \Gamma$ be the general divisor. Note that $\overline{D}$ is irreducible. Indeed, the assumption on the rationality degree implies that not all the curves parametrized by $V$ are irreducible. Since $\Pic X\cong \mathbb{Z}$, any effective divisor is ample. If $\overline{D}$ is not irreducible, then each irreducible component of $\overline{D}$ has positive intersection with each component of the reducible curves parametrized by $V$, leading to a contradiction.

By Bertini Theorem $\overline{D}$ is smooth away from $Z$. Assume that $\overline{D}\in \Gamma$ is singular at a point $z\in Z$. Note that $\dim \varphi_*(\tau^{-1}(z))>0$ (cfr. Diagram \ref{diagram2}).
First we want to show that no component of the curves parametrized by $V$ is contained in $Z$. On the contrary, assume that there exists a curve $C'=C'_1+C'_2$ in the family $V$ such that $x\in C'_1$ and $C'_2 \subset Z$. Since $C'$ is connected and $C'_2\subset \overline{D}$ for any $\overline{D} \in \Gamma$, it follows that $\overline{D}\cdot C'_1\geq 2$ and $\overline{D}\cdot C'_2\leq 0$; a contradiction.
Hence, we can find a point $[C']\in V$ such that, if $C'$ is the corresponding curve on $X$, $C'$ contains the point $z$ and the intersection of $C'$ and $\overline{D}$ is proper. It follows that $\overline{D}\cdot C> 2$, a contradiction.
The same argument can be used to show that all the curves parametrized by $V$ are smooth along $Z$.
\end{proof}
Let $d:=\codim_X Z$ be the codimension of $Z$ in $X$. It follows that we can inductively define the following flag of smooth polarized varieties:
\begin{equation}\label{flag}
(X_{d-1},\overline{D}_{d-1})\subset .... \subset (X_{1},\overline{D}_1)\subset  (X_0,\overline{D}_0)=(X,\overline{D});
\end{equation}
where $X_i$ is the birational transform of the general linear subspace $\p^{n-i}$ and $\overline{D}_i$ is the restriction of the divisor $\overline{D}_{i-1}$ to $X_i$, for $i=1,...,d-1$. Note that the restricted birational map $\phi_{|X_i} : X_i\dashrightarrow\p^{n-i}$ is defined by a linear system $\Gamma_i$ that is the restriction to $X_i$ of $\Gamma_{i-1}\subset | \overline{D}_{i-1} |$.  Note also that $\overline{D}_i\cdot C=2$ for any curve that is the preimage of a line in $\p^{n-i}$.
For any $i=1,...,d-1$, we have the following exact sequence:
\begin{equation}\label{es1}
0\rightarrow \sh_{X_{i-1}}\rightarrow \sh_{X_{i-1}}(\overline{D}_{i-1})\rightarrow\sh_{X_i}(\overline{D}_i)\rightarrow 0.
\end{equation}
Then the linear system $\Gamma_{d-1}$ has a fixed component.
Write $\Gamma_{d-1}=\overline{\Gamma}_{d-1}+F_{d-1}$, where $\overline{\Gamma}_{d-1}$ and $F_{d-1}$ are the movable part and the fixed part of the linear system $\Gamma_{d-1}$ respectively.
Look at the map $\phi_{|X_{d-1}} : X_{d-1}\dashrightarrow\p^{n-d+1}$. Let $V_{d-1}$ be a family of rational curves satisfying the assumption of Theorem \ref{ir} at a point $x$ and inducing $\phi_{|X_{d-1}}$.

\begin{clam}\label{c2} The map $\phi_{|X_{d-1}} : X_{d-1}\dashrightarrow \p^{n-d+1}$ is a morphism. 
\end{clam}
\proof The general divisor of $\Gamma_{d-1}$ can be written as $M_{d-1} + F_{d-1}$. Since $V_{d-1}$ is a covering family of curves and $M_{d-1}$ is big, we have that $M_{d-1}\cdot C\geq 1$, where $[C]\in V_{d-1}$. By the way $F_{d-1}$ has been defined, it follows that $F_{d-1}$ is effective and it intersects properly the general curve parametrized by $V_{d-1}$. In particular also $F_{d-1}\cdot C\geq 1$, where $[C]\in V_{d-1}$. 
From the following chain of equalities 
$$\overline {D}_{d-1}\cdot C=(M_{d-1}+F_{d-1})\cdot C=2=M_{d-1}\cdot C +F_{d-1} \cdot C \ \ \ \ \ \ \ \  \ ({\rm where} \ [C]\in V_{d-1}),$$
it follows that $M_{d-1}\cdot C=1$ and $F_{d-1}\cdot C=1$. 

If $C$ is any reducible curve parametrized by  $V_{d-1}$, write it as $C=C_1+C_2$, where $C_1$ is the component through the point $x$. Since $M_{d-1} \cdot C=1$, it follows that 
$$M_{d-1}\cdot C_1=a\geq 1 \ \ \  and \ \ \ M_{d-1} \cdot C_2=-(a-1).$$
	Looking at the fixed component $F_{d-1}$, we have that $F_{d-1}\cdot C_1 + F_{d-1}\cdot C_2=1$. Up to replacing the family $V_{d-1}$ with a family $V'_{d-1}$ that induces the same birational map $\phi_{|X_{d-1}}$ (see Remark \ref{famcha}), we may assume that $C_1\not\subset F_{d-1}$. This implies that $1 = (M_{d-1}+F_{d-1})\cdot C_1\geq a \geq 1$, from which follows that 
$$F_{d-1}\cdot C_2=M_{d-1}\cdot C_1=1 \ \ \ \ and \ \ \ \ F_{d-1}\cdot C_1=M_{d-1}\cdot C_2=0.$$
Let $Z_{d-1}$ be the base locus of $\phi_{|X_{d-1}} : X_{d-1}\dashrightarrow \p^{n-d+1}$. Note that $M_{d-1}$ is the birational transform of the hyperplane in $\p^{n-d+1}$.

The same argument as in the proof of claim \ref{c1} shows that no component of the curves parametrized by $V_{d-1}$ is contained in $Z_{d-1}$. Indeed, assume that there exists a curve $C'=C'_1+C'_2$ in the family $V_{d-1}$ such that $x\in C'_1$ and $C'_2 \subset Z_{d-1}$. Since $C'$ is connected and $C'_2\subset M_{d-1}$ for any $M_{d-1}\in \Gamma_{d-1}$, it follows that $M_{d-1}\cdot C'_1\geq 2$ and $M_{d-1}\cdot C'_2 < 0$; a contradiction.
Now, let $z$ be a point in the indeterminacy locus $Z_{d-1}$ of $\phi_{|X_{d-1}}$.
Since any member of $| \overline{\Gamma}_{d-1} |$ contains $z$, we can find a point $[C']\in V_{d-1}$ and a divisor $M_{d-1}$ such that, if $C'$ is the corresponding curve on $X_{d-1}$, $C'$ contains the point $z$ and the intersection of $C'$ and $M_{d-1}$ is proper. It follows that $M_{d-1}\cdot C\geq 2$, a contradiction. We conclude that $\phi_{|X_{d-1}} : X_{d-1}\dashrightarrow \p^{n-d+1}$ is a morphism. \endproof

We want to show that the base locus $Z$ of $\Gamma$ is a point. By Proposition \ref{clsu} and remark \ref{subl} we can assume that $\dim X\geq 3$. Note also that, by Lefschetz-Sommese theorem, we have that $\rho (X_i)=1$, for any $i=1,...,n-3$. 

Assume that $1\leq \codim_X Z=d\leq n-2$. In this case, the linear system $\Gamma_{d-1}$ has a fixed component and the restricted map $\phi_{|X_{d-1}}: X_{d-1}\rightarrow \p^{n-d+1}$ is a morphism. Recalling that $\rho (X)=1$, we get that $X_{d-1}\cong \p^{n-d+1}$. A Theorem of B\u{a}descu, (cf.~\cite[Theorem 5.4.10]{BS}), implies that $X_d\cong\mathbb{P}^{n-d}$. Using B\u{a}descu's result inductively, we get that $X\cong \p^n$, a contradiction.

Assume that $\codim_X Z=d=n-1$. In this case, $Z$ is a curve, $X_{n-2}$ is a smooth surface and the restricted map $\phi_{|X_{n-2}}: X_{n-2}\rightarrow \p^{2}$ is a morphism. Any effective divisor $B\in H^0(X_{n-2}, \overline{D}_{n-2})$ can be written as $B=C+Z$, where $C$ belongs to the family $V_{n-2}$. The following relation holds:
\begin{equation}\label{D3}
\overline{D}_{n-3}^3=B^2=(C+Z)^2=C^2+2CZ+Z^2=3+Z^2;
\end{equation}
$$\overline{D}_{n-2}\cdot C=C^2+C\cdot Z=2;$$ 
$$\overline{D}_{n-2}\cdot Z=C\cdot Z+ Z^2>0.$$
In particular, $C^2=C\cdot Z=1$ and $0\leq Z^2\leq 1$. By the Riemann-Roch Theorem for singular curves, we get that
$$h^0(B, B_{|B})=(C+Z)C+(C+Z)Z+1+1-p_a(C+Z)= 2+(1+Z^2)+1=4+Z^2.$$
The exact sequence (\ref{es1}) shows that  $h^0(X_{n-3},\overline{D}_{n-3})=6+Z^2=\overline{D}^3_{n-3}+3$. Since $\overline{D}_{n-3}$ is very ample, the threefold $X_{n-3}$ is a variety of minimal degree in $\p^{Z^2+5}$ (cf. \cite{EH}). By Theorem 1 in \cite{EH}, $X_{n-3}$ is a smooth scroll, a contradiction.

To conclude observe that since the base locus $Z$ is a point, we have $\overline{D}^n=(\overline{D}_{|\overline{D}})^{n-1}=2$. From the short exact sequence (\ref{es1}),
we deduce that $H^0(X,\overline{D}))=n+2$. This implies that $X$ is a smooth variety of minimal degree of codimension 1, that is $X\cong Q^n\subset\p^{n+1}$.
\end{proof}


\begin{thebibliography}{}
\frenchspacing

\bibitem{BS} M.C. Beltrametti, A.J. Sommese. The Adjunction Theory of Complex Projective Varieties.
Expositions in Mathematics, vol. 16, W. de Gruyter, Berlin, 1995.

\bibitem{BB} F. Bogomolov, C. B\"{o}hning. {\it On uniformly rational varieties }, Topology, geometry, integrable systems, and mathematical physics, 33-48, Amer. Math. Soc. Transl. Ser. 2, 234, Amer. Math. Soc., Providence, RI, (2014).

\bibitem{C} A. Corti. {\it Factoring birational maps of threefolds after Sarkisov}, J. Algebraic Geom.,
4(2) (1995), 223-254.

\bibitem{EH} D. Eisenbud, J. Harris. \textit{On varieties of minimal degree (A Centennial Account)}, Proceedings of Symposia in Pure Mathematics, Volume 46 (1987).

\bibitem{G} M. Gromov. {\it Oka's Principle for Holomorphic Sections of Elliptic Bundles},
Journal of the American Mathematical Society, Vol. 2, No. 4 (1989), 851?897.

\bibitem{Har77} R. Hartshorne. Algebraic geometry. Springer-Verlag, New York, 1977. Graduate Texts in Mathematics, No. 52.

\bibitem{H1} J.-M. Hwang, {\it Geometry of minimal rational curves on Fano manifolds} School on Vanishing Theorems and Effective Results in Algebraic Geometry (Trieste, 2000), ICTP Lect. Notes, vol. 6, Abdus Salam Int. Cent. Theoret. Phys., Trieste, 2001, pp. 335?393.

\bibitem{H2} J.-M. Hwang, {\it Rigidity of rational homogeneous spaces}, International Congress of Mathematicians. Vol. II, Eur. Math. Soc., Z\"urich, 2006, pp. 613?626. MR MR2275613 (2008b:14069)

\bibitem{IR} P. Ionescu, F. Russo. \textit{Conic-connected manifolds}, J. Reine Angew. Math., 644 (2010), 145-157.

\bibitem{Ka} I. Karzhemanov, \textit{Around the uniform rationality}, preprint.

\bibitem{kebekus}  S. Kebekus. \textit{Families of singular rational curves}, J. Algebraic Geom. \textbf{11} 2002, 245-256.

\bibitem{Ko} J. Koll\'{a}r, Rational curves on algebraic varieties, Springer-Verlag (1996).

\bibitem{L} R. Lazarsfeld, Positivity in Algebraic Geometry I,II, Ergebnisse der Mathematik und ihrer Grenzgebiete. 3. Folge. A Series of Modern Surveys in Mathematics, 49. Springer-Verlag, Berlin (2004)

\end{thebibliography}
\end{document}